\documentclass[reqno,fleqn]{amsart}
\usepackage{amsmath}
\usepackage{amssymb}
\usepackage{amsfonts}
\usepackage{graphicx}
\usepackage{amsthm}

\makeatletter \@namedef{subjclassname@2010}{  \textup{2010}
Mathematics Subject Classification}
\makeatother
\theoremstyle{plain}
\newtheorem{theorem}{Theorem}[section]
\newtheorem{corollary}[theorem]{Corollary}
\newtheorem{proposition}[theorem]{Proposition}
\newtheorem{lemma}[theorem]{Lemma}
\theoremstyle{remark}

\theoremstyle{definition}

\numberwithin{equation}{section}
\input{tcilatex}

\begin{document}
\title[CONVERGENCE THEOREMS]{SOME\ CONVERGENCE THEOREMS FOR\ OPERATOR\
SEQUENCES}
\author{HEYBETKULU MUSTAFAYEV}
\address{Institute of Mathematics and Mechanics of Academy Sciences of
Azerbaijan, BAKU-AZERBAIJAN}
\email{hsmustafayev@yahoo.com}

\begin{abstract}
Let $A,$ $T$ and $B$ be bounded linear operators on a Banach space. This
paper is concerned mainly with finding some necessary and sufficient
conditions for convergence in operator norm of the sequences $\left\{
A^{n}TB^{n}\right\} $ and $\left\{ \frac{1}{n}\sum_{i=0}^{n-1}A^{i}TB^{i}%
\right\} $. These results are applied to the Toeplitz, composition and model
operators. Some related problems are also discussed.
\end{abstract}

\maketitle

\section{Introduction}

Throughout this paper, $H$ will denote a complex separable infinite
dimensional Hilbert space and $B\left( H\right) $, the algebra of all
bounded linear operators on $H.$ The ideal of all compact operators on $H$
will be denoted by $K\left( H\right) $. The quotient algebra $B\left(
H\right) \diagup K\left( H\right) $ is a $C^{\ast }-$algebra and called the
\textit{Calkin algebra}. As usual, $H^{2}$ will denote the classical Hardy
space on the open unit disk $\mathbb{D}:=\left\{ z\in
%TCIMACRO{\U{2102} }%
%BeginExpansion
\mathbb{C}
%EndExpansion
:\left\vert z\right\vert <1\right\} .$ By $H^{\infty }$ we will denote the
space of all bounded analytic functions on $\mathbb{D}.$

Let $\mathbb{T}$:$=\partial \mathbb{D}$ be the unit circle and let $m$ be
the normalized Lebesgue measure on $\mathbb{T}$. Recall that for a given
symbol $\varphi \in L^{\infty }:=L^{\infty }\left( \mathbb{T},m\right) $,
the \textit{Toeplitz operator} $T_{\varphi }$ on $H^{2}$ is defined by%
\begin{equation*}
T_{\varphi }f=P_{+}\left( \varphi f\right) ,\text{ \ }f\in H^{2},
\end{equation*}%
where $P_{+}$ is the orthogonal projection from $L^{2}\left( \mathbb{T}%
,m\right) $ onto $H^{2}$. Let
\begin{equation*}
Sf\left( z\right) =zf\left( z\right)
\end{equation*}%
be the unilateral shift operator on $H^{2}.$ According to a theorem of Brown
and Halmos \cite{3}, $T\in B\left( H^{2}\right) $ is a Toeplitz operator if
and only if
\begin{equation*}
S^{\ast }TS=T.
\end{equation*}%
Barria and Halmos \cite{1} examined the so-called \textit{strongly
asymptotically Toeplitz operators} $T$ on $H^{2}$ for which the sequence $%
\left\{ S^{\ast n}TS^{n}\right\} $ converges strongly. This class includes
the Hankel algebra, the operator norm-closed algebra generated by all
Toeplitz and Hankel operators together \cite{1}.

An operator $T\in B\left( H^{2}\right) $ is said to be\textit{\ uniformly
asymptotically Toeplitz }if the sequence $\left\{ S^{\ast n}TS^{n}\right\} $
converges in the uniform operator topology. This class of operators is
closed in operator norm and under adjoints. It contains both Toeplitz
operators and the compact ones. Feintuch \cite{8} proved that an operator $%
T\in B\left( H^{2}\right) $ is uniformly asymptotically Toeplitz if and only
if it has the decomposition
\begin{equation*}
T=T_{0}+K,
\end{equation*}%
where $T_{0}$ is a Toeplitz operator, that is, $S^{\ast }T_{0}S=T_{0}$ and $K
$ is a compact operator.

Recall that each holomorphic function $\phi :\mathbb{D\rightarrow D}$
induces a bounded linear composition operator $C_{\phi }$ on $H^{2}$ by $%
C_{\phi }f=f\circ \phi $ (for instance, see \cite[Ch.5]{16}). The only
composition operator, which is also Toeplitz, is the identity operator \cite%
{20}. Using Feintuch's result, Nazarov and Shapiro \cite[Theorem 1.1]{20}
proved that a composition operator on $H^{2}$ is uniformly asymptotically
Toeplitz if and only if it is either compact or the identity operator.

Let $B\left( X\right) $ be the algebra of all bounded linear operators on a
complex Banach space $X$ and let $A,$ $T$ and $B$ be in $B\left( X\right) .$
The main purpose of this paper is to find necessary and sufficient
conditions for convergence in operator norm of the sequences $\left\{
A^{n}TB^{n}\right\} $ and $\left\{ \frac{1}{n}\sum_{i=0}^{n-1}A^{i}TB^{i}%
\right\} $.

\section{The sequence $\left\{ A^{n}TB^{n}\right\} $}

In this section, we give some results concerning convergence in operator
norm of the sequence $\left\{ A^{n}TB^{n}\right\} $ for Hilbert space
operators.

Recall that an operator $T\in B\left( H\right) $ is said to be \textit{%
essentially isometric} (resp. \textit{essentially unitary}) if $I-T^{\ast
}T\in K\left( H\right) $ (resp. $I-T^{\ast }T\in K\left( H\right) $ and $%
I-TT^{\ast }\in K\left( H\right) $).

We have the following:

\begin{theorem}
Let $A$ and $B^{\ast }$ be two essentially isometric operators on $H$ such
that $\left\Vert A^{n}x\right\Vert \rightarrow 0$ and $\left\Vert B^{\ast
n}x\right\Vert \rightarrow 0,$ for all $x\in H$. If $T\in B\left( H\right) ,$
then the sequence $\left\{ A^{n}TB^{n}\right\} $ converges in operator norm
if and only if we have the decomposition
\begin{equation*}
T=T_{0}+K,
\end{equation*}%
where $AT_{0}B=T_{0}$ and $K\in K\left( H\right) .$
\end{theorem}

For the proof, we need some preliminary results.

Let $H_{0}$ be the linear space of all weakly null sequences $\left\{
x_{n}\right\} $ in $H.$ Let us define a semi-inner product in $H_{0}$ by
\begin{equation*}
\langle \left\{ x_{n}\right\} ,\left\{ y_{n}\right\} \rangle =\text{l.i.m.}%
_{n}\langle x_{n},y_{n}\rangle ,
\end{equation*}%
where l.i.m. is a fixed Banach limit. If
\begin{equation*}
E:=\left\{ \left\{ x_{n}\right\} \in H_{0}:\text{ l.i.m.}_{n}\left\Vert
x_{n}\right\Vert ^{2}=0\right\} ,
\end{equation*}%
then $H_{0}\diagup E$ becomes a pre-Hilbert \ space with respect to the
inner product defined by
\begin{equation*}
\langle \left\{ x_{n}\right\} +E,\text{ }\left\{ y_{n}\right\} +E\rangle =%
\text{l.i.m.}_{n}\langle \left\{ x_{n}\right\} ,\left\{ y_{n}\right\}
\rangle .
\end{equation*}%
Let $\widehat{H}$ be the Hilbert space defined by the completion of $%
H_{0}\diagup E$ with respect to the induced norm
\begin{equation*}
\left\Vert \left\{ x_{n}\right\} +E\right\Vert =\left( \text{l.i.m.}%
_{n}\left\Vert x_{n}\right\Vert ^{2}\right) ^{\frac{1}{2}}.
\end{equation*}%
Now, for a given $T\in B\left( H\right) $ we define an operator $\widehat{T}$
on $H_{0}\diagup E$ by
\begin{equation*}
\widehat{T}:\left\{ x_{n}\right\} +E\mapsto \left\{ Tx_{n}\right\} +E.
\end{equation*}%
Consequently, we can write%
\begin{eqnarray*}
\left\Vert \widehat{T}\left( \left\{ x_{n}\right\} +E\right) \right\Vert
&=&\left( \text{l.i.m.}_{n}\left\Vert Tx_{n}\right\Vert ^{2}\right) ^{\frac{1%
}{2}} \\
&\leq &\left\Vert T\right\Vert \left( \text{l.i.m.}_{n}\left\Vert
x_{n}\right\Vert ^{2}\right) ^{\frac{1}{2}} \\
&=&\left\Vert T\right\Vert \left\Vert \left\{ x_{n}\right\} +E\right\Vert .
\end{eqnarray*}%
Since $H_{0}\diagup E$ is dense in $\widehat{H},$ the operator $\widehat{T}$
can be extended to the whole $\widehat{H}$ which we also denote by $\widehat{%
T}$. Clearly, $\left\Vert \widehat{T}\right\Vert \leq \left\Vert
T\right\Vert .$ The operator $\widehat{T}$ will be called \textit{limit
operator associated with} $T.$

\begin{proposition}
If $\widehat{T}$ is the limit operator associated with $T\in B\left(
H\right) $, then:

$\left( a\right) $ The map $T\mapsto \widehat{T}$ is a linear contractive $%
\ast -$homomorphism.

$\left( b\right) $ $T$ is a compact operator if and only if $\widehat{T}=0.$

$\left( c\right) $ $T$ is an \textit{essentially isometry }$($resp. \textit{%
essentially }unitary$)$\textit{\ }if and only if $\widehat{T}$ is an
isometry $($resp. unitary$).$

$\left( d\right) $ For an arbitrary $T\in B\left( H\right) ,$ we have $%
\left\Vert \widehat{T}\right\Vert =\left\Vert T+K\left( H\right) \right\Vert
.$
\end{proposition}

\begin{proof}
Proofs of the assertions (a), (b) and (c) are omitted, since they are clear.
Let us prove (d). Let $\widehat{K}$ be the limit operator associated with $%
K\in K\left( H\right) .$ Since $\widehat{K}=0,$ we get
\begin{equation*}
\left\Vert \widehat{T}\right\Vert =\left\Vert \widehat{T}+\widehat{K}%
\right\Vert \leq \left\Vert T+K\right\Vert \text{, \ }\forall K\in K\left(
H\right) .
\end{equation*}%
This implies $\left\Vert \widehat{T}\right\Vert \leq \left\Vert T+K\left(
H\right) \right\Vert .$ For the reverse inequality, recall \cite[p.94]{2}
that%
\begin{equation*}
\left\Vert T+K\left( H\right) \right\Vert =\sup \left\{ \overline{%
\lim_{n\rightarrow \infty }}\left\Vert Tx_{n}\right\Vert :\left\Vert
x_{n}\right\Vert =1,\text{ }\forall n\in
%TCIMACRO{\U{2115} }%
%BeginExpansion
\mathbb{N}
%EndExpansion
\text{ and }x_{n}\rightarrow 0\text{ weakly}\right\} .
\end{equation*}%
Therefore, for a given $\varepsilon >0$ there exists a sequence $\left\{
x_{n}\right\} $ in $H$ such that $\left\Vert x_{n}\right\Vert =1$ $\left(
\forall n\in
%TCIMACRO{\U{2115} }%
%BeginExpansion
\mathbb{N}
%EndExpansion
\right) ,$ $x_{n}\rightarrow 0$ weakly and
\begin{equation*}
\overline{\lim_{n\rightarrow \infty }}\left\Vert Tx_{n}\right\Vert \geq
\left\Vert T+K\left( H\right) \right\Vert -\varepsilon .
\end{equation*}%
Consequently, there exists a subsequence $\left\{ x_{n_{k}}\right\} $ of $%
\left\{ x_{n}\right\} $ such that
\begin{equation*}
\lim_{k\rightarrow \infty }\left\Vert Tx_{n_{k}}\right\Vert \geq \left\Vert
T+K\left( H\right) \right\Vert -\varepsilon .
\end{equation*}%
On the other hand,
\begin{equation*}
\left\Vert \widehat{T}\right\Vert =\sup \left\{ \left( \text{l.i.m.}%
_{n}\left\Vert Tx_{n}\right\Vert ^{2}\right) ^{\frac{1}{2}}:\text{l.i.m.}%
_{n}\left\Vert x_{n}\right\Vert ^{2}=1\text{ and }x_{n}\rightarrow 0\text{
weakly}\right\} .
\end{equation*}%
As l.i.m.$_{k}\left\Vert x_{n_{k}}\right\Vert ^{2}=1$ and $%
x_{n_{k}}\rightarrow 0$ $\left( k\rightarrow \infty \right) $ weakly, by the
preceding identity we get
\begin{equation*}
\left\Vert \widehat{T}\right\Vert \geq \lim_{k\rightarrow \infty }\left\Vert
Tx_{n_{k}}\right\Vert \geq \left\Vert T+K\left( H\right) \right\Vert
-\varepsilon .
\end{equation*}%
Since $\varepsilon $ is arbitrary, we have $\left\Vert \widehat{T}%
\right\Vert \geq \left\Vert T+K\left( H\right) \right\Vert ,$ as required.
\end{proof}

\begin{lemma}
$\left( a\right) $ Let $A,B\in B\left( H\right) $ and assume that $%
\left\Vert A^{n}x\right\Vert \rightarrow 0$ and $\left\Vert B^{\ast
n}x\right\Vert \rightarrow 0,$ for all $x\in H.$ Then, for an arbitrary $%
K\in K\left( H\right) ,$ we have%
\begin{equation*}
\lim_{n\rightarrow \infty }\left\Vert A^{n}KB^{n}\right\Vert =0.
\end{equation*}

$\left( b\right) $ If $A$ and $B^{\ast }$ are essentially isometric
operators and
\begin{equation*}
\lim_{n\rightarrow \infty }\left\Vert A^{n}TB^{n}\right\Vert =0,
\end{equation*}%
then $T$ is a compact operator.
\end{lemma}

\begin{proof}
(a) For an arbitrary $x,y\in H,$ let $x\otimes y$ be the rank one operator
on $H$;%
\begin{equation*}
x\otimes y:z\mapsto \langle z,y\rangle x,\text{ \ }z\in H.
\end{equation*}%
Since finite rank operators are dense (in operator norm) in $K\left(
H\right) $, we may assume that $K$ is a finite rank operator, say,%
\begin{equation*}
K=\sum_{i=1}^{N}x_{i}\otimes y_{i},
\end{equation*}%
where $x_{i},y_{i}\in H$ $\left( i=1,...,N\right) .$ Consequently, we can
write%
\begin{equation*}
\left\Vert A^{n}KB^{n}\right\Vert =\left\Vert
\sum_{i=1}^{N}A^{n}x_{i}\otimes B^{\ast n}y_{i}\right\Vert \leq
\sum_{i=1}^{N}\left\Vert A^{n}x_{i}\right\Vert \left\Vert B^{\ast
n}y_{i}\right\Vert \rightarrow 0\text{ \ }\left( n\rightarrow \infty \right)
.
\end{equation*}

(b) Let $\widehat{A}$, $\widehat{T}$ and $\widehat{B}$ be the limit
operators associated with $A$, $T$ and $B,$ respectively. By Proposition 2.2$%
,$ $\widehat{A}$ and $\widehat{B}^{\ast }$ are isometries. Since the map $%
T\mapsto \widehat{T}$ is a contractive homomorphism, for an arbitrary $n\in
%TCIMACRO{\U{2115} }%
%BeginExpansion
\mathbb{N}
%EndExpansion
$ we get
\begin{equation*}
\left\Vert \widehat{T}\right\Vert =\left\Vert \widehat{A}^{n}\widehat{T}%
\widehat{B}^{n}\right\Vert \leq \left\Vert A^{n}TB^{n}\right\Vert
\rightarrow 0.
\end{equation*}%
Hence $\widehat{T}=0.$ By Proposition 2.2, $T$ is a compact operator.
\end{proof}

We are now in a position to prove Theorem 2.1.

\begin{proof}[Proof of Theorem 2.1]
If $T=T_{0}+K,$ where $AT_{0}B=T_{0}$ and $K\in K\left( H\right) ,$ then
\begin{equation*}
A^{n}TB^{n}=T_{0}+A^{n}KB^{n}\text{, \ }\forall n\in
%TCIMACRO{\U{2115} }%
%BeginExpansion
\mathbb{N}
%EndExpansion
.
\end{equation*}%
By Lemma 2.3, $\left\Vert A^{n}KB^{n}\right\Vert \rightarrow 0$ and
therefore $\left\Vert A^{n}TB^{n}-T_{0}\right\Vert \rightarrow 0.$ Now,
assume that there exists $T_{0}\in B\left( H\right) $ such that $\left\Vert
A^{n}TB^{n}-T_{0}\right\Vert \rightarrow 0.$\ Since
\begin{equation*}
\left\Vert A^{n+1}TB^{n+1}-AT_{0}B\right\Vert \rightarrow 0,
\end{equation*}%
we have $AT_{0}B=T_{0}$ which implies $A^{n}T_{0}B^{n}=T_{0}$ for all $n\in
%TCIMACRO{\U{2115} }%
%BeginExpansion
\mathbb{N}
%EndExpansion
.$ Also, since
\begin{equation*}
\left\Vert A^{n}(T-T_{0})B^{n}\right\Vert \rightarrow 0,
\end{equation*}%
by Lemma 2.3, $T-T_{0}$ is a compact operator. So we have $T=T_{0}+K,$ where
$K\in K\left( H\right) .$
\end{proof}

As a consequence of Theorem 2.1 we have the following:

\begin{corollary}
Let $A\in B\left( H\right) $ and assume that $I-AA^{\ast }\in K\left(
H\right) $ and $\left\Vert A^{\ast n}x\right\Vert \rightarrow 0$ for all $%
x\in H.$ If $T\in B\left( H\right) ,$ then the sequence $\left\{ A^{\ast
n}TA^{n}\right\} $ converges in operator norm if and only if we have the
decomposition $T=T_{0}+K,$ where $A^{\ast }T_{0}A=T_{0}$ and $K\in K\left(
H\right) .$
\end{corollary}

If $S$ is the unilateral shift on $H^{2}$, then the operator $I-SS^{\ast }$
is one dimensional and $\left\Vert S^{\ast n}f\right\Vert \rightarrow 0$ for
all $f\in H^{2}.$ By taking $A=S$ in Corollary 2.4, we obtain Feintuch's
result mentioned above.

Let an arbitrary $\varphi ,\psi \in L^{\infty }$ be given. As we have noted
in the Introduction, $T_{\varphi }T_{\psi }$ is a strongly asymptotically
Toeplitz operator, that is, $S^{\ast n}T_{\varphi }T_{\psi }S^{n}\rightarrow
T_{\varphi \psi }$ strongly \cite[Theorem 4]{1}. From this and from
Corollary 2.4 it follows that $T_{\varphi }T_{\psi }$ is a uniformly
asymptotically Toeplitz operator if and only if $T_{\varphi }T_{\psi }$ is a
compact perturbation of the Toeplitz operator $T_{\varphi \psi }.$ Now,
assume that one of the functions $\varphi ,\psi $ is a trigonometric
polynomial, say, $\psi =\sum_{-N}^{N}c_{k}e^{ik\theta }.$ Then as
\begin{equation*}
T_{\psi }=\sum_{k=1}^{N}c_{-k}S^{\ast k}+\sum_{k=0}^{N}c_{k}S^{k},
\end{equation*}%
$S^{\ast n}T_{\varphi }S^{\ast k}S^{n}=S^{\ast k}T_{\varphi }$ $\left(
\forall n\geq k\right) $ and $S^{\ast n}T_{\varphi }S^{k}S^{n}=T_{\varphi
}S^{k}$ $\left( \forall k\geq 0\right) ,$ we have%
\begin{equation*}
S^{\ast n}T_{\varphi }T_{\psi }S^{n}=\sum_{k=1}^{N}c_{-k}S^{\ast
k}T_{\varphi }+\sum_{k=0}^{N}c_{k}T_{\varphi }S^{k},\text{ \ }\forall n\geq
N.
\end{equation*}%
If $\varphi =\sum_{-N}^{N}c_{k}e^{ik\theta },$ then as $S^{\ast n}S^{\ast
k}T_{\psi }S^{n}=S^{\ast k}T_{\psi }$ $\left( \forall k\geq 0\right) $ and $%
S^{\ast n}S^{k}T_{\psi }S^{n}=T_{\psi }S^{k}$ $\left( \forall n\geq k\right)
,$ we have
\begin{equation*}
S^{\ast n}T_{\varphi }T_{\psi }S^{n}=\sum_{k=1}^{N}c_{-k}S^{\ast k}T_{\psi
}+\sum_{k=0}^{N}c_{k}T_{\psi }S^{k},\text{ \ }\forall n\geq N.
\end{equation*}%
Therefore, if one of the functions $\varphi ,\psi $ is continuous$,$ then $%
T_{\varphi }T_{\psi }$ is a uniformly asymptotically Toeplitz operator.
Further, if $\psi $ has the form $\psi =h+f$, where $h\in H^{\infty }$ and $%
f\in C\left( \mathbb{T}\right) ,$ then as $T_{\varphi }T_{h}=T_{\varphi h}$
we get%
\begin{eqnarray*}
S^{\ast n}T_{\varphi }T_{\psi }S^{n} &=&S^{\ast n}T_{\varphi }\left(
T_{h}+T_{f}\right) S^{n} \\
&=&S^{\ast n}T_{\varphi h}S^{n}+S^{\ast n}T_{\varphi }T_{f}S^{n} \\
&=&T_{\varphi h}+S^{\ast n}T_{\varphi }T_{f}S^{n}.
\end{eqnarray*}%
It follows that $T_{\varphi }T_{\psi }$ is a uniformly asymptotically
Toeplitz operator for all $\varphi \in L^{\infty }$ and $\psi \in H^{\infty
}+C\left( \mathbb{T}\right) $ (recall that the algebraic sum $H^{\infty
}+C\left( \mathbb{T}\right) $ is a uniformly closed subalgebra of $L^{\infty
}$ and sometimes called a \textit{Douglas algebra}). Consequently, $%
T_{\varphi }T_{\psi }$ is a compact perturbation of the Toeplitz operator $%
T_{\varphi \psi }$ for all $\varphi \in L^{\infty }$ and $\psi \in H^{\infty
}+C\left( \mathbb{T}\right) .$ Similarly, we can see that if $\varphi $ has
the form $\varphi =\overline{h}+f$, where $h\in H^{\infty }$ and $f\in
C\left( \mathbb{T}\right) ,$ then $T_{\varphi }T_{\psi }$ is a uniformly
asymptotically Toeplitz operator.

Note that in Corollary 2.4, compactness condition of the operator $%
I-AA^{\ast }$ is essential. To see this, let $A=V$ be the Volterra integral
operator on $H=L^{2}\left[ 0,1\right] .$ Then, $I-VV^{\ast }\notin K\left(
H\right) $ and as $\left\Vert V^{n}\right\Vert \rightarrow 0,$ we have $%
\left\Vert V^{\ast n}x\right\Vert \rightarrow 0$ for all $x\in H.$ Since $%
\left\Vert V^{\ast n}TV^{n}\right\Vert \rightarrow 0$ for all $T\in B\left(
H\right) ,$ the equation $V^{\ast }T_{0}V=T_{0}$ has only zero solution. If
the conclusion of Corollary 2.4 were true, we would get $B\left( H\right)
\subseteq K\left( H\right) ,$ which is a contradiction.

Let $H^{2}\left( E\right) $ be the Hardy space of all analytic functions on $%
\mathbb{D}$ with values in a Hilbert space $E.$ Let $A\in B\left( H\right) $
be a contraction, $E:=\overline{\left( I-AA^{\ast }\right) H}$ and assume
that $\left\Vert A^{\ast n}x\right\Vert \rightarrow 0$ for all $x\in H.$ By
the Model Theorem of Nagy-Foia\c{s} (see, \cite[Ch.VI, Theorem 2.3]{19} and
\cite{21}), $A$ is unitary equivalent to its model operator
\begin{equation*}
A_{\Theta }f:=P_{\mathcal{K}}S_{E}f,\text{ \ }f\in \mathcal{K},
\end{equation*}%
where $\mathcal{K}=H^{2}\left( E\right) \ominus \Theta H^{2}\left( F\right)
, $ $F$ is a subspace of $E,$ $\Theta $ is a bounded analytic function on $%
\mathbb{D}$ with values in $B\left( F,E\right) ,$ the space of all bounded
linear operators from $F$ into $E$ ($\Theta \left( \xi \right) $ is an
isometry for almost all $\xi \in \mathbb{T}$)$,$ $P_{\mathcal{K}}$ is the
orthogonal projection from $H^{2}\left( E\right) $ onto $\mathcal{K}$ and $%
S_{E}$ is the unilateral shift operator on $H^{2}\left( E\right) .$ Notice
also that $A_{\Theta }^{\ast }=S_{E}^{\ast }\mid _{\mathcal{K}}$.
Consequently, Corollary 2.4 can be applied to the model operator $A_{\Theta
} $ in the case when the operator $A$ satisfies the following conditions: 1)
$A $ is a contraction; 2) $\left\Vert A^{\ast n}x\right\Vert \rightarrow 0$
for all $x\in H;$ 3) The defect operator $\mathcal{D}_{A^{\ast
}}:=(I-AA^{\ast })^{\frac{1}{2}}$ is compact.

In addition, assume that $\left\Vert A^{n}x\right\Vert \rightarrow 0$ for
all $x\in H.$ In this case, the subspace $E$ can be identified with $F$ and $%
\Theta \left( \xi \right) $ becomes unitary for almost all $\xi \in \mathbb{T%
}$. Consequently, Proposition 2.5 (shown below) is applicable to the model
operator $A_{\Theta }$ in the case when the operator $A$ satisfies the
following conditions: 1) $A$ is a contraction; 2) $\left\Vert
A^{n}x\right\Vert \rightarrow 0$ and $\left\Vert A^{\ast n}x\right\Vert
\rightarrow 0$ for all $x\in H;$ 3) the defect operator $\mathcal{D}%
_{A^{\ast }}$ is compact.

\begin{proposition}
Let $A\in B\left( H\right) $ and assume that $I-AA^{\ast }\in K\left(
H\right) ,$ $\left\Vert A^{n}x\right\Vert \rightarrow 0$ and $\left\Vert
A^{\ast n}x\right\Vert \rightarrow 0$ for all $x\in H.$ For an arbitrary $%
T\in B\left( H\right) ,$ the following assertions are equivalent:

$\left( a\right) $ The sequence $\left\{ A^{\ast n}TA^{n}\right\} $
converges in operator norm.

$\left( b\right) $ $A^{\ast n}TA^{n}\rightarrow 0$ in operator norm.

$\left( c\right) $ $T$ is a compact operator.
\end{proposition}

\begin{proof}
(a)$\Rightarrow $(b) By Corollary 2.4, $T=T_{0}+K$, where $A^{\ast
}T_{0}A=T_{0}$ and $K\in K\left( H\right) .$ On the other hand, by Lemma
2.3, $\left\Vert A^{\ast n}KA^{n}\right\Vert \rightarrow 0$. It remains to
show that $T_{0}=0.$ Indeed, for an arbitrary $x,y\in H,$ from the identity $%
A^{\ast n}T_{0}A^{n}=T_{0}$ $\left( \forall n\in
%TCIMACRO{\U{2115} }%
%BeginExpansion
\mathbb{N}
%EndExpansion
\right) $, we can write%
\begin{equation*}
\left\vert \langle T_{0}x,y\rangle \right\vert =\left\vert \langle
T_{0}A^{n}x,A^{n}y\rangle \right\vert \leq \left\Vert T_{0}\right\Vert
\left\Vert A^{n}x\right\Vert \left\Vert A^{n}y\right\Vert \rightarrow 0.
\end{equation*}%
Hence $T_{0}=0.$

(b)$\Rightarrow $(c)$\Rightarrow $(a) are obtained from Lemma 2.3.
\end{proof}

Recall that an operator $T\in B\left( X\right) $ is said to be \textit{%
almost periodic} if for every $x\in X$, the orbit $\left\{ T^{n}x:n\in
%TCIMACRO{\U{2115} }%
%BeginExpansion
\mathbb{N}
%EndExpansion
\right\} $ is relatively compact. Clearly, an almost periodic operator is
power bounded, that is,
\begin{equation*}
\sup_{n\geq 0}\left\Vert T^{n}\right\Vert <\infty .
\end{equation*}%
If $T\in B\left( X\right) $ is an almost periodic operator, then by the
Jacobs-Glicksberg-de Leeuw decomposition theorem \cite[Ch.I, Theorem 1.15]{7}%
, every $x\in X$ can be written as $x=x_{0}+x_{1},$ where $\left\Vert
T^{n}x_{0}\right\Vert \rightarrow 0$ and $x_{1}\in \overline{\text{span}}%
\left\{ y\in X:\exists \xi \in \mathbb{T},\text{ }Ty=\xi y\right\} .$

From now on, for a given $T\in B\left( X\right) $ the left and right
multiplication operators on $B\left( X\right) $ will be denoted by $L_{T}$
and $R_{T},$ respectively.

The following result is an improvement of Proposition 2.5.

\begin{proposition}
Let $A\in B\left( H\right) $ and assume that $I-AA^{\ast }\in K\left(
H\right) ,$ $\left\Vert A^{n}x\right\Vert \rightarrow 0$ and $\left\Vert
A^{\ast n}x\right\Vert \rightarrow 0$ for all $x\in H.$ For an arbitrary $%
T\in B\left( H\right) $ the following assertions are equivalent:

$\left( a\right) $ $\left\{ A^{\ast n}TA^{n}:n\in
%TCIMACRO{\U{2115} }%
%BeginExpansion
\mathbb{N}
%EndExpansion
\right\} $ is relatively compact in the operator norm topology.

$\left( b\right) $ $A^{\ast n}TA^{n}\rightarrow 0$ in operator norm.

$\left( c\right) $ $T$ is a compact operator.
\end{proposition}

\begin{proof}
(a)$\Rightarrow $(b) Let $E$ be the set of all $Q\in B\left( H\right) $ such
that
\begin{equation*}
\left\{ \left( L_{A^{\ast }}R_{A}\right) ^{n}Q:n\in
%TCIMACRO{\U{2115} }%
%BeginExpansion
\mathbb{N}
%EndExpansion
\right\}
\end{equation*}
is relatively compact in the operator norm topology. By the uniform
boundedness principle, the operator $L_{A^{\ast }}R_{A}$ is power bounded
and therefore $E$ is a closed (in operator norm) $L_{A^{\ast }}R_{A}-$%
invariant subspace. Consequently, $L_{A^{\ast }}R_{A}\mid _{E},$ the
restriction of $L_{A^{\ast }}R_{A}$ to $E$ is an almost periodic operator.
Since $T\in E,$ by the Jacobs-Glicksberg-de Leeuw decomposition theorem, $%
T=T_{0}+T_{1},$ where
\begin{equation*}
\lim_{n\rightarrow \infty }\left\Vert A^{\ast n}T_{0}A^{n}\right\Vert =0
\end{equation*}%
and%
\begin{equation*}
T_{1}\in \overline{\text{span}}^{\left\Vert \cdot \right\Vert }\left\{ Q\in
E:\exists \xi \in \mathbb{T},\text{ }A^{\ast }QA=\xi Q\right\} .
\end{equation*}%
We must show that $T_{1}=0.$ For this, it suffices to show that the identity
$A^{\ast }QA=\xi Q$ $\left( \xi \in \mathbb{T}\right) $ implies $Q=0.$
Indeed, since
\begin{equation*}
A^{\ast n}QA^{n}=\xi ^{n}Q\text{, }\forall n\in
%TCIMACRO{\U{2115} }%
%BeginExpansion
\mathbb{N}
%EndExpansion
,
\end{equation*}%
we get%
\begin{equation*}
\left\vert \langle Qx,y\rangle \right\vert =\left\vert \langle
QA^{n}x,A^{n}y\rangle \right\vert \leq \left\Vert Q\right\Vert \left\Vert
A^{n}x\right\Vert \left\Vert A^{n}y\right\Vert \rightarrow 0\text{, \ }%
\forall x,y\in H.
\end{equation*}%
Hence $Q=0$.

(b)$\Rightarrow $(c)$\Rightarrow $(a) are obtained from Lemma 2.3.
\end{proof}

Next, we have the following:

\begin{theorem}
Let $A$ and $B^{\ast }$ be two essentially isometric contractions on $H$ and
assume that $\left\Vert A^{n}x\right\Vert \rightarrow 0$ and $\left\Vert
B^{\ast n}x\right\Vert \rightarrow 0$ for all $x\in H.$ Then, for an
arbitrary $T\in B\left( H\right) $ we have%
\begin{equation*}
\lim_{n\rightarrow \infty }\left\Vert A^{n}TB^{n}\right\Vert =\left\Vert
T+K\left( H\right) \right\Vert .
\end{equation*}
\end{theorem}

\begin{proof}
If $K\in K\left( H\right) ,$ then by Lemma 2.3, $\left\Vert
A^{n}KB^{n}\right\Vert \rightarrow 0.$ Since
\begin{equation*}
\left\Vert A^{n}\left( T+K\right) B^{n}\right\Vert \leq \left\Vert
T+K\right\Vert ,
\end{equation*}%
we have
\begin{equation*}
\lim_{n\rightarrow \infty }\left\Vert A^{n}TB^{n}\right\Vert \leq \left\Vert
T+K\left( H\right) \right\Vert .
\end{equation*}%
For the reverse inequality, let $\widehat{A}$, $\widehat{T}$ and $\widehat{B}
$ be the limit operators associated with $A$, $T$ and $B,$ respectively. By
Proposition 2.2, $\widehat{A}$ and $\widehat{B}^{\ast }$ are isometries. By
using the same proposition again, we can write%
\begin{equation*}
\left\Vert T+K\left( H\right) \right\Vert =\left\Vert \widehat{T}\right\Vert
=\left\Vert \widehat{A}^{n}\widehat{T}\widehat{B}^{n}\right\Vert \leq
\left\Vert A^{n}TB^{n}\right\Vert \text{, \ }\forall n\in
%TCIMACRO{\U{2115} }%
%BeginExpansion
\mathbb{N}
%EndExpansion
.
\end{equation*}%
Thus we have
\begin{equation*}
\left\Vert T+K\left( H\right) \right\Vert \leq \lim_{n\rightarrow \infty
}\left\Vert A^{n}TB^{n}\right\Vert .
\end{equation*}
\end{proof}

We know \cite[Corollary 7.13]{5} that every Toeplitz operator $T_{\varphi }$
with symbol $\varphi \in L^{\infty }$ satisfies%
\begin{equation*}
\left\Vert T_{\varphi }\right\Vert =\left\Vert T_{\varphi }+K\left(
H^{2}\right) \right\Vert .
\end{equation*}

As a consequence of Theorem 2.7, we have the following generalization of the
preceding formula.

\begin{corollary}
Let $A\in B\left( H\right) $ be a contraction and assume that $I-AA^{\ast
}\in K\left( H\right) $ and $\left\Vert A^{\ast n}x\right\Vert \rightarrow 0$
for all $x\in H.$ Then, for an arbitrary $T\in B\left( H\right) $ we have%
\begin{equation*}
\lim_{n\rightarrow \infty }\left\Vert A^{\ast n}TA^{n}\right\Vert
=\left\Vert T+K\left( H\right) \right\Vert .
\end{equation*}
\end{corollary}

For an arbitrary $A,B\in B\left( H\right) ,$ we put
\begin{equation*}
\mathcal{I}_{A,B}=\left\{ T\in B\left( H\right) :ATB=T\right\} .
\end{equation*}

\begin{proposition}
Assume that the operators $A,B$ satisfy the hypotheses of Theorem 2.7. Then,
for an arbitrary $K\in K\left( H\right) $ we have%
\begin{equation*}
\left\Vert K+\mathcal{I}_{A,B}\right\Vert \geq \frac{1}{2}\left\Vert
K\right\Vert .
\end{equation*}%
In the case $AB=I,$ this estimate is the best possible.
\end{proposition}

\begin{proof}
Assume that there exists $K\in K\left( H\right) $ such that
\begin{equation*}
\left\Vert K+\mathcal{I}_{A,B}\right\Vert <\frac{1}{2}\left\Vert
K\right\Vert .
\end{equation*}%
Then there exists $T\in \mathcal{I}_{A,B}$ such that
\begin{equation*}
\left\Vert K+T\right\Vert <\frac{1}{2}\left\Vert K\right\Vert .
\end{equation*}%
By Theorem 2.7, $\left\Vert T\right\Vert =\left\Vert T+K\left( H\right)
\right\Vert $ which implies $\left\Vert T\right\Vert \leq \left\Vert
K+T\right\Vert $. Consequently, we can write%
\begin{equation*}
\left\Vert K\right\Vert \leq \left\Vert K+T\right\Vert +\left\Vert
T\right\Vert \leq 2\left\Vert K+T\right\Vert <\left\Vert K\right\Vert ,
\end{equation*}%
which is a contradiction.

In the case $AB=I,$ we have $I\in \mathcal{I}_{A,B}.$ If $K=x\otimes x,$
where $\left\Vert x\right\Vert =1$ and $T=-\frac{1}{2}I$, then $\left\Vert
K+T\right\Vert =\frac{1}{2}.$
\end{proof}

Let $\mathcal{T}$ be the space of all Toeplitz operators. By taking $%
A=S^{\ast }$ and $B=S$ in Proposition 2.9, we have
\begin{equation*}
\left\Vert K+\mathcal{T}\right\Vert \geq \frac{1}{2}\left\Vert K\right\Vert
\text{, \ }\forall K\in K\left( H^{2}\right) ,
\end{equation*}%
where this estimate is the best possible.

\section{One dimensional model and the Hartman-Sarason theorem}

Using the results of the preceding section, here we give a quantitative
generalization of the Hartman-Sarason theorem.

Recall that a contraction $T$ on $H$ is said to be \textit{completely
non-unitary} if it has no proper reducing subspace on which it acts as a
unitary operator. If $T$ is a completely non-unitary contraction, then $%
f\left( T\right) $ $\left( f\in H^{\infty }\right) $ can be defined by the
Nagy-Foia\c{s} functional calculus \cite[Ch.III]{19}.

Let $T$ be a contraction on $H$ and assume that
\begin{equation*}
\lim_{n\rightarrow \infty }\left\Vert T^{n}x\right\Vert =\lim_{n\rightarrow
\infty }\left\Vert T^{\ast n}x\right\Vert =0\text{, \ }\forall x\in H.
\end{equation*}%
In addition, if%
\begin{equation*}
\dim \left( I-TT^{\ast }\right) H=\dim \left( I-T^{\ast }T\right) H=1,
\end{equation*}%
then by the Model Theorem of Nagy-Foia\c{s} \cite[Ch.VI, Theorem 2.3]{19}
(see also, \cite{21}) $T$ is unitary equivalent to its \textit{model operator%
}
\begin{equation*}
S_{\theta }=P_{\theta }S\mid _{H_{\theta }^{2}}
\end{equation*}%
acting on the \textit{model space}
\begin{equation*}
H_{\theta }^{2}=H^{2}\ominus \theta H^{2},
\end{equation*}%
where $\theta $ is an inner function (a function $\theta $ in $H^{\infty }$
is an inner function if $\left\vert \theta \right\vert =1$ a.e. on $\mathbb{T%
}$) and $P_{\theta }$ is the orthogonal projection from $H^{2}$ onto $%
H_{\theta }^{2}$. Beurling's theorem \cite[Corollary 6.11]{5} says that
these spaces are generic invariant subspaces for the backward shift operator
\begin{equation*}
\left( S^{\ast }f\right) \left( z\right) =\frac{f\left( z\right) -f\left(
0\right) }{z},\text{ \ }f\in H^{2}.
\end{equation*}%
Notice that
\begin{equation*}
S_{\theta }=\left( S^{\ast }\mid _{H_{\theta }^{2}}\right) ^{\ast }.
\end{equation*}

Let $\theta $ be an inner function and let $S_{\theta }$ be the model
operator on the model space $H_{\theta }^{2}.$ For an arbitrary $f\in
H^{\infty }$, we can define the operator
\begin{equation*}
f\left( S_{\theta }\right) =P_{\theta }f\left( S\right) \mid _{H_{\theta
}^{2}}
\end{equation*}%
which is unitary equivalent to $f\left( T\right) .$ The map $f\mapsto
f\left( S_{\theta }\right) $ is linear, multiplicative and by the Nehari
formula \cite[p.235]{21},%
\begin{equation*}
\left\Vert f\left( S_{\theta }\right) \right\Vert =\text{dist}\left(
\overline{\theta }f,H^{\infty }\right) .
\end{equation*}%
Let us mention Sarason's theorem \cite[p.230]{21} which asserts that an
operator $Q\in B\left( H_{\theta }^{2}\right) $ is a commutant of $S_{\theta
}$ if and only if $Q=f\left( S_{\theta }\right) $ for some $f\in H^{\infty }$%
.

Let us also mention that the classical theorem of Hartman and Sarason \cite[%
p.235]{21} classifies compactness of the operators $f\left( S_{\theta
}\right) .$ The operator $f\left( S_{\theta }\right) $ $\left( f\in
H^{\infty }\right) $ is compact if and only if $\overline{\theta }f\in
H^{\infty }+C\left( \mathbb{T}\right) .$

We have the following quantitative generalization of the Hartman-Sarason
theorem.

\begin{theorem}
Let $\theta $ be an inner function and let $S_{\theta }$ be the model
operator on the model space $H_{\theta }^{2}.$ Then, for an arbitrary $f\in
H^{\infty }$ we have%
\begin{equation*}
\left\Vert f\left( S_{\theta }\right) +K\left( H_{\theta }^{2}\right)
\right\Vert =\textnormal{dist}\left( \overline{\theta }f,H^{\infty }+C\left(
\mathbb{T}\right) \right) .
\end{equation*}
\end{theorem}

For the proof, we need several lemmas.

\begin{lemma}
Let $\left\{ E_{n}\right\} $ be an increasing sequence of closed subspaces
of a Banach space $X.$ Then, for an arbitrary $x\in X$ we have
\begin{equation*}
\lim_{n\rightarrow \infty }\textnormal{dist}\left( x,E_{n}\right) =%
\textnormal{dist}\left( x,\overline{\bigcup_{n=1}^{\infty }E_{n}}\right)
\text{.}
\end{equation*}
\end{lemma}

\begin{proof}
If $x\in X,$ then the sequence $\left\{ \text{dist}\left( x,E_{n}\right)
\right\} $ is decreasing. Let
\begin{equation*}
\alpha :=\lim_{n\rightarrow \infty }\text{dist}\left( x,E_{n}\right)
=\inf_{n}\text{dist}\left( x,E_{n}\right) .
\end{equation*}%
Since
\begin{equation*}
E_{n}\subseteq \overline{\bigcup_{n=1}^{\infty }E_{n}},
\end{equation*}%
we have
\begin{equation*}
\text{dist}\left( x,\overline{\bigcup_{n=1}^{\infty }E_{n}}\right) \leq
\text{dist}\left( x,E_{n}\right)
\end{equation*}%
which implies%
\begin{equation*}
\text{dist}\left( x,\overline{\bigcup_{n=1}^{\infty }E_{n}}\right) \leq
\alpha .
\end{equation*}%
If
\begin{equation*}
\text{dist}\left( x,\overline{\bigcup_{n=1}^{\infty }E_{n}}\right) <\alpha ,
\end{equation*}%
then $\left\Vert x-x_{0}\right\Vert <\alpha $ for some $x_{0}\in
\bigcup_{n=1}^{\infty }E_{n}.$ Consequently, $x_{0}\in E_{n_{0}}$ for some $%
n_{0}.$ Hence dist$\left( x,E_{n_{0}}\right) <\alpha .$ This contradicts dist%
$\left( x,E_{n_{0}}\right) \geq \alpha .$
\end{proof}

\begin{lemma}
For an arbitrary $\varphi \in L^{\infty }$ we have%
\begin{equation*}
\lim_{n\rightarrow \infty }\textnormal{dist}\left( \varphi ,\overline{z}%
^{n}H^{\infty }\right) =\textnormal{dist}\left( \varphi ,H^{\infty }+C\left(
\mathbb{T}\right) \right) .
\end{equation*}
\end{lemma}

\begin{proof}
We know \cite[Proposition 6.36]{5} that $H^{\infty }+C\left( \mathbb{T}%
\right) $ is a uniformly closed subalgebra\textit{\ }of $L^{\infty }$
generated by $\overline{z}$ and $H^{\infty }.$ If $E_{n}:=\overline{z}%
^{n}H^{\infty },$ then $\left\{ E_{n}\right\} $ is an increasing sequence of
closed subspaces of $L^{\infty }$. Since%
\begin{equation*}
H^{\infty }+C\left( \mathbb{T}\right) =\overline{\text{span}}_{L^{\infty
}}\left\{ \overline{z}^{n}H^{\infty }:n\geq 0\right\}
\end{equation*}%
and%
\begin{equation*}
\overline{z}^{n}f_{1}+\overline{z}^{m}f_{2}=\left(
z^{m}f_{1}+z^{n}f_{2}\right) \overline{z}^{n+m}\in \overline{z}%
^{n+m}H^{\infty }\text{ \ }\left( f_{1},f_{2}\in H^{\infty }\right) ,
\end{equation*}%
we have
\begin{equation*}
\overline{\bigcup_{n=1}^{\infty }E_{n}}=H^{\infty }+C\left( \mathbb{T}%
\right) .
\end{equation*}%
Applying Lemma 3.2 to the subspaces $\left\{ E_{n}\right\} ,$ we obtain our
result.
\end{proof}

Now, we can prove Theorem 3.1.

\begin{proof}[Proof of Theorem 3.1]
As we have noted above, the model operator $S_{\theta }$ is an essentially
unitary contraction. Moreover, $\left\Vert S_{\theta }^{n}h\right\Vert
\rightarrow 0$ and $\left\Vert S_{\theta }^{\ast n}h\right\Vert \rightarrow 0
$ for all $h\in H_{\theta }^{2}.$ If $T\in B\left( H_{\theta }^{2}\right) ,$
then by Theorem 2.7,
\begin{equation*}
\lim_{n\rightarrow \infty }\left\Vert S_{\theta }^{\ast n}TS_{\theta
}^{n}\right\Vert =\left\Vert T+K\left( H_{\theta }^{2}\right) \right\Vert
\end{equation*}%
which implies
\begin{equation*}
\lim_{n\rightarrow \infty }\left\Vert TS_{\theta }^{n}\right\Vert \geq
\left\Vert T+K\left( H_{\theta }^{2}\right) \right\Vert .
\end{equation*}%
If $K\in K\left( H_{\theta }^{2}\right) ,$ then as $\left\Vert KS_{\theta
}^{n}\right\Vert \rightarrow 0$ (see, the proof of Lemma 2.3) we get
\begin{equation*}
\lim_{n\rightarrow \infty }\left\Vert TS_{\theta }^{n}\right\Vert
=\lim_{n\rightarrow \infty }\left\Vert \left( T+K\right) S_{\theta
}^{n}\right\Vert \leq \left\Vert T+K\right\Vert ,\text{ \ }\forall K\in
K\left( H_{\theta }^{2}\right) .
\end{equation*}%
It follows that
\begin{equation*}
\lim_{n\rightarrow \infty }\left\Vert TS_{\theta }^{n}\right\Vert \leq
\left\Vert T+K\left( H_{\theta }^{2}\right) \right\Vert .
\end{equation*}%
Thus we have%
\begin{equation*}
\lim_{n\rightarrow \infty }\left\Vert TS_{\theta }^{n}\right\Vert
=\left\Vert T+K\left( H_{\theta }^{2}\right) \right\Vert \text{, \ }\forall
T\in B\left( H_{\theta }^{2}\right) .
\end{equation*}%
In particular, taking $T=f\left( S_{\theta }\right) $ we obtain
\begin{equation}
\lim_{n\rightarrow \infty }\left\Vert S_{\theta }^{n}f\left( S_{\theta
}\right) \right\Vert =\left\Vert f\left( S_{\theta }\right) +K\left(
H_{\theta }^{2}\right) \right\Vert .  \label{3.1}
\end{equation}%
Further, by the Nehari formula we can write
\begin{equation*}
\left\Vert S_{\theta }^{n}f\left( S_{\theta }\right) \right\Vert =\text{dist}%
\left( \overline{\theta }z^{n}f,H^{\infty }\right) =\text{dist}\left(
\overline{\theta }f,\overline{z}^{n}H^{\infty }\right) .
\end{equation*}%
On the other hand, by Lemma 3.3,
\begin{equation*}
\lim_{n\rightarrow \infty }\left\Vert S_{\theta }^{n}f\left( S_{\theta
}\right) \right\Vert =\lim_{n\rightarrow \infty }\text{dist}\left( \overline{%
\theta }f,\overline{z}^{n}H^{\infty }\right) =\text{dist}\left( \overline{%
\theta }f,H^{\infty }+C\left( \mathbb{T}\right) \right) .
\end{equation*}%
Now, taking into account (3.1), finally we obtain%
\begin{equation*}
\left\Vert f\left( S_{\theta }\right) +K\left( H_{\theta }^{2}\right)
\right\Vert =\text{dist}\left( \overline{\theta }f,H^{\infty }+C\left(
\mathbb{T}\right) \right) .
\end{equation*}%
The proof is complete.
\end{proof}

Below, we present some applications of Theorem 3.1.

Let $X$ be a Banach space. As usual, $\sigma \left( T\right) $ will denote
the spectrum of the operator $T\in B\left( X\right) .$ Given $T\in B\left(
X\right) ,$ we let $A_{T}$ denote the closure in the uniform operator
topology of all polynomials in $T.$ Then, $A_{T}$ is a commutative unital
Banach algebra. The Gelfand space of $A_{T}$ can be identified with $\sigma
_{A_{T}}\left( T\right) $, the spectrum of $T$ with respect to the algebra $%
A_{T}$. Since $\sigma \left( T\right) $ is a (closed) subset of $\sigma
_{A_{T}}\left( T\right) ,$ for every $\lambda \in \sigma \left( T\right) ,$
there is a multiplicative functional $\phi _{\lambda }$ on $A_{T}$ such that
$\phi _{\lambda }\left( T\right) =\lambda $. By $\widehat{Q}$ we will denote
the Gelfand transform of $Q\in A_{T}$. Instead of $\widehat{Q}\left( \phi
_{\lambda }\right) \left( =\phi _{\lambda }\left( Q\right) \right) ,$ where $%
\lambda \in \sigma \left( T\right) ,$ we will use the notation $\widehat{Q}%
\left( \lambda \right) .$ It follows from the Shilov Theorem \cite[Theorem
2.54]{5} that if $T$ is a contraction, then
\begin{equation*}
\sigma _{A_{T}}\left( T\right) \cap \mathbb{T=}\sigma \left( T\right) \cap
\mathbb{T}\text{.}
\end{equation*}

The following result was obtained in \cite{17}.

\begin{theorem}
If $T$ is a contraction on a Hilbert space, then for an arbitrary $Q\in
A_{T} $ we have%
\begin{equation*}
\lim_{n\rightarrow \infty }\left\Vert T^{n}Q\right\Vert =\sup_{\xi \in
\sigma \left( T\right) \cap \mathbb{T}}\left\vert \widehat{Q}\left( \xi
\right) \right\vert .
\end{equation*}
\end{theorem}

For a non-empty closed subset $\Gamma $ of $\mathbb{T}$, by $H_{\Gamma
}^{\infty }$ we will denote the set of all those functions $f$ in $H^{\infty
}$ that have a continuous extension $\widetilde{f}$ to $\mathbb{D}\cup
\Gamma $. Clearly, $H_{\Gamma }^{\infty }$ is a closed subspace of $%
H^{\infty }.$ It follows from the general theory of $H^{p}$ spaces that if $%
\Gamma $ has positive Lebesgue measure and $f\in H_{\Gamma }^{\infty }$ is
not identically zero, then $\widetilde{f}$ cannot vanish identically on $%
\Gamma .$

If $T$ is a contraction on a Hilbert space $H,$ then there is a canonical
decomposition of $H$ into two $T-$reducing subspaces $H=H_{0}\oplus H_{u}$
such that $T_{0}:=T\mid _{H_{0}}$ is completely non-unitary and $%
T_{u}:=T\mid _{H_{u}}$ is unitary \cite[Ch.I, Theorem 3.2]{19}. It can be
seen that%
\begin{equation*}
\sigma \left( T_{u}\right) \subseteq \sigma \left( T\right) \cap \mathbb{T}.
\end{equation*}%
Let $f$ be in $H_{\sigma \left( T\right) \cap \mathbb{T}}^{\infty }$ with
continuous extension $\widetilde{f}$ to $\mathbb{D}\cup (\sigma \left(
T\right) \cap \mathbb{T)}$. As in \cite{10}, we can define $f\left( T\right)
\in B\left( H\right) $ by%
\begin{equation*}
f\left( T\right) =f\left( T_{0}\right) \oplus \widetilde{f}\left(
T_{u}\right) ,
\end{equation*}%
where $f\left( T_{0}\right) $ is given by the Nagy-Foias functional calculus
and
\begin{equation*}
\widetilde{f}\left( T_{u}\right) =\left( \widetilde{f}\mid _{\sigma \left(
T\right) \cap \mathbb{T}}\right) \left( T_{u}\right) .
\end{equation*}%
It can be seen that
\begin{equation*}
\left\Vert f\left( T\right) \right\Vert \leq \left\Vert f\right\Vert
_{\infty }\text{, \ }\forall f\in H_{\sigma \left( T\right) \cap \mathbb{T}%
}^{\infty }.
\end{equation*}%
Further, by the Gamelin-Garnett theorem \cite{9}, there exists a sequence $%
\left\{ f_{n}\right\} $ in $H^{\infty }$ such that each $f_{n}$ has an
analytic extension $g_{n}$ to a neighborhood $O_{n}$ of $\mathbb{D}\cup
(\sigma \left( T\right) \cap \mathbb{T)}$ and%
\begin{equation*}
\lim_{n\rightarrow \infty }\left\Vert f_{n}-f\right\Vert _{\infty }=0.
\end{equation*}%
Then, $g_{n}\left( T\right) $ can be defined by the Riesz-Dunford functional
calculus. Since $f_{n}\left( T\right) =g_{n}\left( T\right) \in A_{T}$ and
\begin{equation*}
\left\Vert f_{n}\left( T\right) -f\left( T\right) \right\Vert \leq
\left\Vert f_{n}-f\right\Vert _{\infty }\rightarrow 0,
\end{equation*}%
we have that $f\left( T\right) \in A_{T}$. Moreover,
\begin{equation*}
\widehat{f\left( T\right) }\left( \xi \right) =\widetilde{f}\left( \xi
\right) \text{, \ }\forall \xi \in \sigma \left( T\right) \cap \mathbb{T}.
\end{equation*}

As a consequence of Theorem 3.4 we have the following:

\begin{corollary}
Let $T$ be a contraction on a Hilbert space. If $f\in H_{\sigma \left(
T\right) \cap \mathbb{T}}^{\infty }$ with continuous extension $\widetilde{f}
$ to $\mathbb{D}\cup (\sigma \left( T\right) \cap \mathbb{T)},$ then
\begin{equation*}
\lim_{n\rightarrow \infty }\left\Vert T^{n}f\left( T\right) \right\Vert
=\sup_{\xi \in \sigma \left( T\right) \cap \mathbb{T}}\left\vert \widetilde{f%
}\left( \xi \right) \right\vert .
\end{equation*}
\end{corollary}

Now, let $\theta $ be an inner function and let $S_{\theta }$ be the model
operator on the model space $H_{\theta }^{2}.$ We put

\begin{equation*}
\Sigma _{u}\left( \theta \right) =\left\{ \xi \in \mathbb{T}:\underset{z\in
\mathbb{D},\text{ }z\rightarrow \xi }{\lim \inf }\left\vert \theta \left(
z\right) \right\vert =0\right\} .
\end{equation*}%
It follows from the Lipschitz-Moeller theorem \cite[p.81]{21} that
\begin{equation*}
\sigma \left( S_{\theta }\right) \cap \mathbb{T}=\Sigma _{u}\left( \theta
\right) .
\end{equation*}%
If $f\in H_{\Sigma _{u}\left( \theta \right) }^{\infty }$ with continuous
extension $\widetilde{f}$ to $\mathbb{D}\cup \Sigma _{u}\left( \theta
\right) ,$ then by Corollary 3.5,
\begin{equation*}
\lim_{n\rightarrow \infty }\left\Vert S_{\theta }^{n}f\left( S_{\theta
}\right) \right\Vert =\sup_{\xi \in \Sigma _{u}\left( \theta \right)
}\left\vert \widetilde{f}\left( \xi \right) \right\vert \text{.}
\end{equation*}%
On the other hand, by (3.1),
\begin{equation*}
\lim_{n\rightarrow \infty }\left\Vert S_{\theta }^{n}f\left( S_{\theta
}\right) \right\Vert =\left\Vert f\left( S_{\theta }\right) +K\left(
H_{\theta }^{2}\right) \right\Vert .
\end{equation*}%
Thus we have
\begin{equation*}
\left\Vert f\left( S_{\theta }\right) +K\left( H_{\theta }^{2}\right)
\right\Vert =\sup_{\xi \in \Sigma _{u}\left( \theta \right) }\left\vert
\widetilde{f}\left( \xi \right) \right\vert \text{.}
\end{equation*}%

From Theorem 3.1 and from the preceding identity we have the following:

\begin{corollary}
Let $\theta $ be an inner function and let $S_{\theta }$ be the model
operator on the model space $H_{\theta }^{2}.$ For an arbitrary $f\in
H_{\Sigma _{u}\left( \theta \right) }^{\infty }$ with continuous extension $%
\widetilde{f}$ to $\mathbb{D}\cup \Sigma _{u}\left( \theta \right) ,$ we have%
\begin{equation*}
\left\Vert f\left( S_{\theta }\right) +K\left( H_{\theta }^{2}\right)
\right\Vert =\textnormal{dist}\left( \overline{\theta }f,H^{\infty }+C\left(
\mathbb{T}\right) \right) =\sup_{\xi \in \Sigma _{u}\left( \theta \right)
}\left\vert \widetilde{f}\left( \xi \right) \right\vert .
\end{equation*}
\end{corollary}

\section{The sequence $\left\{ \frac{1}{n}\sum_{i=0}^{n-1}A^{i}TB^{i}\right%
\} $}

In this section, we give some results concerning convergence in operator
norm of the sequence $\left\{ \frac{1}{n}\sum_{i=0}^{n-1}A^{i}TB^{i}\right\}
$ for Hilbert space operators.

Let $X$ be a Banach space. It is easy to check that if $T\in B\left(
X\right) $ is power bounded, then
\begin{equation*}
\overline{\left( T-I\right) X}=\left\{ x\in X:\lim_{n\rightarrow \infty
}\left\Vert \frac{1}{n}\sum_{i=0}^{n-1}T^{i}x\right\Vert =0\right\} .
\end{equation*}

The following result is well known (for instance, see \cite[Ch.2, \S 2.1,
Theorems 1.2 and 1.3]{12}).

\begin{proposition}
Let $T\in B\left( X\right) $ be power bounded and let $E$ be the set of all $%
x\in X$ such that the sequence $\left\{ \frac{1}{n}\sum_{i=0}^{n-1}T^{i}x%
\right\} $ converges strongly. Then, we have the decomposition%
\begin{equation*}
E=\overline{\left( T-I\right) X}\oplus \ker \left( T-I\right) .
\end{equation*}%
If $X$ is reflexive, then $E=X.$
\end{proposition}

Applying Proposition 4.1 to the operator $L_{A}R_{B}$ on the space $B\left(
X\right) $, we have the following:

\begin{corollary}
Let $A,B\in B\left( X\right) $ be two operators such that $\sup_{n\geq
0}\left( \left\Vert A^{n}\right\Vert \left\Vert B^{n}\right\Vert \right)
<\infty $ and $T\in B\left( X\right) .$ Then, the sequence $\left\{ \frac{1}{%
n}\sum_{i=0}^{n-1}A^{i}TB^{i}\right\} $ converges in operator norm if and
only if we have the decomposition $T=T_{0}+Q,$ where
\begin{equation*}
\left\Vert \frac{1}{n}\sum_{i=0}^{n-1}A^{i}T_{0}B^{i}\right\Vert \rightarrow
0\text{ and }AQB=Q.
\end{equation*}
\end{corollary}

\begin{lemma}
Let $T\in B\left( X\right) $ be power bounded, $x\in X$ and assume that
\begin{equation*}
\lim_{n\rightarrow \infty }\left\Vert T^{n+1}x-T^{n}x\right\Vert =0.
\end{equation*}%
$\left( a\right) $ If the sequence $\left\{ \frac{1}{n}%
\sum_{i=0}^{n-1}T^{i}x\right\} $ converges strongly, then the sequence $%
\left\{ T^{n}x\right\} $ converges strongly $($to same element$),$ too.

$\left( b\right) $ If $X$ is reflexive, then the sequence $\left\{
T^{n}x\right\} $ converges strongly.
\end{lemma}

\begin{proof}
(a) Notice that
\begin{equation*}
F:=\left\{ y\in X:\lim_{n\rightarrow \infty }\left\Vert
T^{n+1}y-T^{n}y\right\Vert =0\right\}
\end{equation*}%
is a closed $T-$invariant subspace and $x\in F.$ Since $T$ is power bounded
and%
\begin{equation*}
\left\Vert T^{n}(T-I)y\right\Vert =\left\Vert T^{n+1}y-T^{n}y\right\Vert
\rightarrow 0\text{, \ }\forall y\in F,
\end{equation*}%
we have $\left\Vert T^{n}y\right\Vert \rightarrow 0$ for all $y\in \overline{%
\left( T-I\right) F}.$ Now, let $E$ be the set of all $y\in F$ such that the
sequence $\left\{ \frac{1}{n}\sum_{i=0}^{n-1}T^{i}y\right\} $ converges
strongly. Since $x\in E,$ by Proposition 4.1 we have the decomposition $%
x=x_{0}+y_{0},$ where $x_{0}\in \overline{\left( T-I\right) F}$ and $%
Ty_{0}=y_{0}.$ As $T^{n}x=T^{n}x_{0}+y_{0}$ and $\left\Vert
T^{n}x_{0}\right\Vert \rightarrow 0,$ we have $\left\Vert
T^{n}x-y_{0}\right\Vert \rightarrow 0.$ Clearly, $\frac{1}{n}%
\sum_{i=0}^{n-1}T^{i}x\rightarrow y_{0}$ strongly.

(b) If $X$ is reflexive, then by Proposition 4.1 the sequence $\left\{ \frac{%
1}{n}\sum_{i=0}^{n-1}T^{i}x\right\} $ converges strongly for every $x\in X.$
By (a), the sequence $\left\{ T^{n}x\right\} $ converges strongly.
\end{proof}

Next, we have the following:

\begin{theorem}
Let $A$ and $B^{\ast }$ be two essentially isometric operators on $H$ and $%
T\in B\left( H\right) .$ Assume that:

$\left( i\right) \ \left\Vert A^{n}x\right\Vert \rightarrow 0$ and $%
\left\Vert B^{\ast n}x\right\Vert \rightarrow 0$ for all $x\in H;$

$\left( ii\right) $ $ATB-T\in K\left( H\right) .$

Then, the sequence $\left\{ \frac{1}{n}\sum_{i=0}^{n-1}A^{i}TB^{i}\right\} $
converges in operator norm if and only if we have the decomposition $%
T=T_{0}+K,$ where $AT_{0}B=T_{0}$ and $K\in K\left( H\right) .$
\end{theorem}

\begin{proof}
Assume that the sequence $\left\{ \frac{1}{n}\sum_{i=0}^{n-1}A^{i}TB^{i}%
\right\} $ converges in operator norm. Since $ATB-T\in K\left( H\right) ,$
by Lemma 2.3,%
\begin{equation*}
\lim_{n\rightarrow \infty }\left\Vert \left( L_{A}R_{B}\right)
^{n+1}T-\left( L_{A}R_{B}\right) ^{n}T\right\Vert =\lim_{n\rightarrow \infty
}\left\Vert A^{n}\left( ATB-T\right) B^{n}\right\Vert =0.
\end{equation*}%
Notice also that the operator $L_{A}R_{B}$ is power bounded. Applying Lemma
4.3 to the operator $L_{A}R_{B}$ on the space $B\left( X\right) $, we obtain
that the sequence $\left\{ A^{n}TB^{n}\right\} $ converges in operator norm.
By Theorem 2.1, $T=T_{0}+K,$ where $AT_{0}B=T_{0}$ and $K\in K\left(
H\right) .$

If $T=T_{0}+K,$ where $AT_{0}B=T_{0}$ and $K\in K\left( H\right) ,$ then we
have
\begin{equation*}
\frac{1}{n}\sum_{i=0}^{n-1}A^{i}TB^{i}=T_{0}+\frac{1}{n}%
\sum_{i=0}^{n-1}A^{i}KB^{i}.
\end{equation*}%
By Lemma 2.3, $\left\Vert A^{n}KB^{n}\right\Vert \rightarrow 0$ and
therefore $\left\Vert \frac{1}{n}\sum_{i=0}^{n-1}A^{i}KB^{i}\right\Vert
\rightarrow 0.$ Thus
\begin{equation*}
\frac{1}{n}\sum_{i=0}^{n-1}A^{i}TB^{i}\rightarrow T_{0}\text{ in operator
norm.}
\end{equation*}
\end{proof}

\begin{corollary}
Assume that the operators $A,T\in B\left( H\right) $ satisfy the following
conditions:

$\left( i\right) $ $I-AA^{\ast }\in K\left( H\right) ;$

$\left( ii\right) $ $\left\Vert A^{\ast n}x\right\Vert \rightarrow 0$ for
all $x\in H;$

$\left( iii\right) $ $A^{\ast }TA-T\in K\left( H\right) .$

Then, the sequence $\left\{ \frac{1}{n}\sum_{i=0}^{n-1}A^{\ast
i}TA^{i}\right\} $ converges in operator norm if and only if we have the
decomposition $T=T_{0}+K,$ where $A^{\ast }T_{0}A=T_{0}$ and $K\in K\left(
H\right) .$
\end{corollary}

The proof of the following lemma is straightforward and will be omitted.

\begin{lemma}
Let $U$ be an essentially unitary operator on $H.$ Then, $T\in B\left(
H\right) $ is an essential commutant of $U$ if and only if $U^{\ast }TU-T\in
K\left( H\right) .$
\end{lemma}

Recall that $T\in B\left( H^{2}\right) $ is an \textit{essentially Toeplitz
operator }if\textit{\ }%
\begin{equation*}
S^{\ast }TS-T\in K\left( H^{2}\right) .
\end{equation*}%
By Lemma 4.6, $T\in B\left( H^{2}\right) $ is an essentially Toeplitz
operator if and only if $T$ is an essential commutant of the unilateral
shift operator $S.$ On the other hand, essential commutant of the unilateral
shift is a $C^{\ast }-$algebra. Consequently, the set of all essentially
Toeplitz operators is a $C^{\ast }-$algebra and therefore contains the $%
C^{\ast }-$algebra generated by all Toeplitz operators.

\begin{corollary}
An \textit{essentially Toeplitz operator }$T$ is a compact perturbation of a
Toeplitz operator if and only if the sequence $\left\{ \frac{1}{n}%
\sum_{i=0}^{n-1}S^{\ast i}TS^{i}\right\} $ converges in operator norm.
\end{corollary}

In \cite{20}, it was proved that if the composition operator $C_{\phi }$ on $%
H^{2}$ is neither compact nor the identity, then $C_{\phi }$ cannot be
compact perturbation of a Toeplitz operator.

\begin{corollary}
If $C_{\phi }$ is a composition operator on $H^{2},$ then the sequence
\begin{equation*}
\left\{ \frac{1}{n}\sum_{i=0}^{n-1}S^{\ast i}C_{\phi }S^{i}\right\}
\end{equation*}
converges in operator norm if and only if either $C_{\phi }$ is compact or
the identity operator.
\end{corollary}

Recall that the class of compact composition operators are sufficiently
large (for instance, see \cite{16}).

Following \cite{15}, we could define an \textit{asymptotic Toeplitz operator
in the Calkin} \textit{algebra} as an operator $T\in B\left( H^{2}\right) $
such that the sequence $\left\{ S^{\ast n}TS^{n}\right\} $ converges in the
Calkin algebra.

The following result, which seems to be unnoticed (see, \cite[p.745]{15}).

\begin{proposition}
Every \textit{asymptotic Toeplitz operator in the Calkin} \textit{algebra}
is an essentially \textit{Toeplitz operator.}
\end{proposition}

\begin{proof}
If $T\in B\left( H^{2}\right) $ is an asymptotic Toeplitz operator in the
Calkin algebra,\textit{\ }then there is an operator $Q\in B\left(
H^{2}\right) $ such that
\begin{equation*}
\lim_{n\rightarrow \infty }\left\Vert S^{\ast n}TS^{n}-Q+K\left( H\right)
\right\Vert =0.
\end{equation*}%
Let $\widehat{S^{\ast }}$, $\widehat{T}$, $\widehat{S}$ and $\widehat{Q}$ be
the limit operators associated with $S^{\ast }$, $T$, $S$ and $Q,$
respectively. By Proposition 2.2,
\begin{equation*}
\lim_{n\rightarrow \infty }\left\Vert \widehat{S^{\ast }}^{n}\widehat{T}%
\widehat{S}^{n}-\widehat{Q}\right\Vert =0.
\end{equation*}%
Since
\begin{equation*}
\lim_{n\rightarrow \infty }\left\Vert \widehat{S^{\ast }}^{n+1}\widehat{T}%
\widehat{S}^{n+1}-\widehat{S^{\ast }}\widehat{Q}\widehat{S}\right\Vert =0,
\end{equation*}%
we have $\widehat{S^{\ast }}\widehat{Q}\widehat{S}=\widehat{Q}.$ By using
the same proposition again, we obtain that $S^{\ast }QS-Q\in K\left(
H^{2}\right) .$
\end{proof}

\section{Banach space operators}

In this section, we study convergence in operator norm of the sequence $%
\left\{ A^{n}TB^{n}\right\} $ for Banach space operators.

Let $X$ be a Banach space. For an arbitrary $T\in B\left( X\right) $ and $%
x\in X$, we define $\rho _{T}\left( x\right) $ to be the set of all $\lambda
\in
%TCIMACRO{\U{2102} }%
%BeginExpansion
\mathbb{C}
%EndExpansion
$ for which there exists a neighborhood $U_{\lambda }$ of $\lambda $ with $%
u\left( z\right) $ analytic on $U_{\lambda }$ having values in $X$ such that
\begin{equation*}
\left( zI-T\right) u\left( z\right) =x\text{, \ }\forall z\in U_{\lambda }.
\end{equation*}%
This set is open and contains the resolvent set $\rho \left( T\right) $ of $T
$. By definition, the \textit{local spectrum} of $T$ at $x\in X$, denoted by
$\sigma _{T}\left( x\right) ,$ is the complement of $\rho _{T}\left(
x\right) $, so it is a compact subset of $\sigma \left( T\right) $. This
object is the most tractable if the operator $T$ has the \textit{%
single-valued extension property }(SVEP), i.e., for every open set $U$ in $%
%TCIMACRO{\U{2102} }%
%BeginExpansion
\mathbb{C}
%EndExpansion
,$ the only analytic function $u:U\rightarrow X$ for which the equation $%
\left( zI-T\right) u\left( z\right) =0$ holds is the constant function $%
u\equiv 0$. If $T$ has SVEP, then $\sigma _{T}\left( x\right) \neq \emptyset
,$ whenever $x\in X\diagdown \left\{ 0\right\} $ \cite[Proposition 1.2.16]%
{13}. Note that the local spectrum of $T$ may be "very small" with respect
to its usual spectrum. To see this, let $\sigma $ be a "small" clopen part
of $\sigma \left( T\right) $. Let $P_{\sigma }$ be the spectral projection
associated with $\sigma $ and $X_{\sigma }:=P_{\sigma }X$. Then, $X_{\sigma }
$ is a closed $T-$invariant subspace of $X$ and $\sigma \left( T\mid
_{X_{\sigma }}\right) =\sigma $. It is easy to see that $\sigma _{T}\left(
x\right) \subseteq \sigma $ for every $x\in X_{\sigma }$.

If $T$ is power bounded, then clearly, $\sigma \left( T\right) \subset
\overline{\mathbb{D}\text{ }}$and $\sigma _{T}\left( x\right) \cap \mathbb{T}
$ consists of all $\xi \in \mathbb{T}$ such that the function $z\rightarrow
\left( zI-T\right) ^{-1}x$ $\left( \left\vert z\right\vert >1\right) $ has
no analytic extension to a neighborhood of $\xi $.

\begin{lemma}
Let $T\in B\left( X\right) $, $x\in X$ and assume that $\sup_{n\geq
0}\left\Vert T^{n}x\right\Vert <\infty .$ Then, $\sigma _{T}\left( x\right)
\subseteq \overline{\mathbb{D}}.$
\end{lemma}

\begin{proof}
Consider the function
\begin{equation*}
u\left( z\right) :=\sum_{n=0}^{\infty }\frac{T^{n}x}{z^{n+1}}
\end{equation*}%
which is analytic on $%
%TCIMACRO{\U{2102} }%
%BeginExpansion
\mathbb{C}
%EndExpansion
\diagdown \overline{\mathbb{D}}$ and $\left( zI-T\right) u\left( z\right) =x$
for all $z\in
%TCIMACRO{\U{2102} }%
%BeginExpansion
\mathbb{C}
%EndExpansion
\diagdown \overline{\mathbb{D}}.$ This shows that $%
%TCIMACRO{\U{2102} }%
%BeginExpansion
\mathbb{C}
%EndExpansion
\diagdown \overline{\mathbb{D}}\subseteq \rho _{T}\left( x\right) $ and
therefore $\sigma _{T}\left( x\right) \subseteq \overline{\mathbb{D}}.$
\end{proof}

We mention the following classical result of Katznelson and Tzafriri \cite[%
Theorem 1]{11}: If $T\in B\left( X\right) $ is power bounded, then $%
\lim_{n\rightarrow \infty }\left\Vert T^{n+1}-T^{n}\right\Vert =0$ if and
only if $\sigma \left( T\right) \cap \mathbb{T}\subseteq \left\{ 1\right\} .$

We have the following local version of the Katznelson-Tzafriri theorem \cite[%
Theorem 4.2]{18}.

\begin{theorem}
Let $T\in B\left( X\right) $, $x\in X$ and assume that $\sup_{n\geq
0}\left\Vert T^{n}x\right\Vert <\infty .$ If $\sigma _{T}\left( x\right)
\cap \mathbb{T}\subseteq \left\{ 1\right\} ,$ then%
\begin{equation*}
\lim_{n\rightarrow \infty }\left\Vert T^{n+1}x-T^{n}x\right\Vert =0.
\end{equation*}
\end{theorem}

Note that in contrast with the Katznelson-Tzafriri theorem, the converse of
Theorem 5.2 does not hold, in general. Indeed, if $S^{\ast }$ is the
backward shift operator on $H^{2},$ then as $\left\Vert S^{\ast
n}f\right\Vert \rightarrow 0,$ we have
\begin{equation*}
\lim_{n\rightarrow \infty }\left\Vert S^{\ast \left( n+1\right) }f-S^{\ast
n}f\right\Vert =0\text{, \ }\forall f\in H^{2}.
\end{equation*}%
On the other hand, since
\begin{equation*}
\left( \lambda I-S^{\ast }\right) ^{-1}f\left( z\right) =\frac{\lambda
^{-1}f\left( \lambda ^{-1}\right) -zf\left( z\right) }{1-\lambda z}\text{ \ }%
\left( \left\vert \lambda \right\vert >1\right) ,
\end{equation*}%
$\sigma _{S^{\ast }}\left( f\right) \cap \mathbb{T}$ consists of all $\xi
\in \mathbb{T}$ for which the function $f$ has no analytic extension to a
neighborhood of $\xi $ (see, \cite[p.24]{6}).

Theorem 5.2 combined with Lemma 4.3 yields the next result.

\begin{theorem}
Assume that $T\in B\left( X\right) $ and $x\in X$ satisfy the following
conditions:

$\left( i\right) $ $\sup_{n\geq 0}\left\Vert T^{n}x\right\Vert <\infty $;

$\left( ii\right) $ $\sigma _{T}\left( x\right) \cap \mathbb{T}\subseteq
\left\{ 1\right\} .$

If the sequence $\left\{ \frac{1}{n}\sum_{i=0}^{n-1}T^{i}x\right\} $
converges strongly to $y\in X,$ then $T^{n}x\rightarrow y$ strongly.
\end{theorem}

\begin{corollary}
Let $T\in B\left( X\right) $ and let $x\in X$ be such that $\sup_{n\geq
0}\left\Vert T^{n}x\right\Vert <\infty .$ Let
\begin{equation*}
S:=\frac{I+T+...+T^{k-1}}{k}\text{ }(k>1\text{ is a fixed integer})
\end{equation*}
and assume that the sequence $\left\{ \frac{1}{n}\sum_{i=0}^{n-1}S^{i}x%
\right\} $ converges strongly to $y\in X.$ Then, $S^{n}x\rightarrow y$
strongly.
\end{corollary}

\begin{proof}
It is easy to check that
\begin{equation*}
\sup_{n\geq 0}\left\Vert S^{n}x\right\Vert \leq \sup_{n\geq 0}\left\Vert
T^{n}x\right\Vert <\infty .
\end{equation*}%
Notice also that if
\begin{equation*}
f\left( z\right) :=\frac{1+z+...+z^{k-1}}{k}\text{ }\left( z\in
%TCIMACRO{\U{2102} }%
%BeginExpansion
\mathbb{C}
%EndExpansion
\right) ,
\end{equation*}%
then $f\left( 1\right) =1$ and $\left\vert f\left( z\right) \right\vert <1$
for all $z\in \overline{\mathbb{D}}\diagdown \left\{ 1\right\} .$ On the
other hand, by \cite[Theorem 3.3.8]{13},
\begin{equation*}
\sigma _{S}\left( x\right) =\sigma _{f\left( T\right) }\left( x\right)
=f\left( \sigma _{T}\left( x\right) \right) .
\end{equation*}%
Since $\sigma _{T}\left( x\right) \subseteq \overline{\mathbb{D}}$ (Lemma
5.1), we have $\sigma _{S}\left( x\right) \cap \mathbb{T}\subseteq \left\{
1\right\} .$ By Theorem 5.3, $S^{n}x\rightarrow y$ strongly.
\end{proof}

We put%
\begin{equation*}
D_{+}=\left\{ z\in
%TCIMACRO{\U{2102} }%
%BeginExpansion
\mathbb{C}
%EndExpansion
:\text{Re}z\geq 1,\text{ Im}z\geq 0\right\} \text{ and }D_{-}=\left\{ z\in
%TCIMACRO{\U{2102} }%
%BeginExpansion
\mathbb{C}
%EndExpansion
:\text{Re}z\geq 1,\text{ Im}z\leq 0\right\} .
\end{equation*}

As another application of Theorem 5.3, we have the following:

\begin{theorem}
Assume that the operators $A,T,B\in B\left( X\right) $ satisfy the following
conditions:

$\left( i\right) \sup_{n\geq 0}\left\Vert A^{n}TB^{n}\right\Vert <\infty ;$

$\left( ii\right) $ either $\sigma \left( A\right) \subset D_{+}$ and $%
\sigma \left( B\right) \subset D_{-}$ or $\sigma \left( A\right) \subset
D_{-}$ and $\sigma \left( B\right) \subset D_{+}.$

If the sequence $\left\{ \frac{1}{n}\sum_{i=0}^{n-1}A^{i}TB^{i}\right\} $
converges in operator norm to $Q\in B\left( X\right) ,$ then $%
A^{n}TB^{n}\rightarrow Q$ in operator norm.
\end{theorem}

\begin{proof}
Since
\begin{equation*}
\sup_{n\geq 0}\left\Vert \left( L_{A}R_{B}\right) ^{n}T\right\Vert
=\sup_{n\geq 0}\left\Vert A^{n}TB^{n}\right\Vert <\infty ,
\end{equation*}%
by Lemma 5.1,%
\begin{equation*}
\sigma _{L_{A}R_{B}}\left( T\right) \subseteq \overline{\mathbb{D}}.
\end{equation*}%
On the other hand, by the Lumer-Rosenblum theorem \cite[Theorem 10]{14},
\begin{equation*}
\sigma \left( L_{A}R_{B}\right) =\left\{ \lambda \mu :\lambda \in \sigma
\left( A\right) ,\text{ }\mu \in \sigma \left( B\right) \right\}
\end{equation*}%
which implies
\begin{equation*}
\sigma _{L_{A}R_{B}}\left( T\right) \subseteq \sigma \left(
L_{A}R_{B}\right) \subset \left\{ z\in
%TCIMACRO{\U{2102} }%
%BeginExpansion
\mathbb{C}
%EndExpansion
:\text{Re}z\geq 1\right\} .
\end{equation*}%
Thus we have%
\begin{equation*}
\sigma _{L_{A}R_{B}}\left( T\right) \subseteq \overline{\mathbb{D}}\cap
\left\{ z\in
%TCIMACRO{\U{2102} }%
%BeginExpansion
\mathbb{C}
%EndExpansion
:\text{Re}z\geq 1\right\} =\left\{ 1\right\} .
\end{equation*}%
Applying Theorem 5.3 to the operator $L_{A}R_{B}$ on the space $B\left(
X\right) $, we obtain that%
\begin{equation*}
A^{n}TB^{n}=\left( L_{A}R_{B}\right) ^{n}T\rightarrow Q\text{ in operator
norm.}
\end{equation*}
\end{proof}

Next, we will show that the hypothesis $\sigma _{T}\left( x\right) \cap
\mathbb{T}\subseteq \left\{ 1\right\} $ in Theorem 5.3 is the best possible,
in general.

Let $N$ be a normal operator on a Hilbert space $H$ with the spectral
measure $P$ and $x\in H.$ Define a measure $\mu _{x}$ on $\sigma \left(
N\right) $ by
\begin{equation}
\mu _{x}\left( \Delta \right) =\langle P\left( \Delta \right) x,x\rangle
=\left\Vert P\left( \Delta \right) x\right\Vert ^{2}\text{.}  \label{5.1}
\end{equation}%
It follows from the Spectral Theorem that $\sigma \left( N\right) =$supp$P$
and $\sigma _{N}\left( x\right) =$supp$\mu _{x}$. It is easy to check that
if $N$ is a contraction (a normal operator is power bounded if and only if
it is a contraction) then,
\begin{equation}
\frac{1}{n}\sum_{i=0}^{n-1}N^{i}x\rightarrow P\left( \left\{ 1\right\}
\right) x\text{ in norm for all }x\in H.  \label{5.2}
\end{equation}

\begin{proposition}
Let $N$ be a normal contraction operator on $H$ with the spectral measure $P$
and $x\in H.$ The sequence $\left\{ N^{n}x\right\} $ converges strongly if
and only if
\begin{equation*}
P\left( \sigma _{N}\left( x\right) \cap \mathbb{T}\diagdown \left\{
1\right\} \right) x=0.
\end{equation*}%
In this case, $N^{n}x\rightarrow P\left( \left\{ 1\right\} \right) x$\
strongly.
\end{proposition}

\begin{proof}
Let $\mu _{x}$ be the measure on $\sigma \left( N\right) $ defined by (5.1).
We can write
\begin{eqnarray*}
&&\lim_{n\rightarrow \infty }\left\Vert N^{n+1}x-N^{n}x\right\Vert
^{2}=\lim_{n\rightarrow \infty }\int_{\sigma _{N}\left( x\right) }\left\vert
z^{n+1}-z^{n}\right\vert ^{2}d\mu _{x}\left( z\right) \\
&=&\lim_{n\rightarrow \infty }\int_{\sigma _{N}\left( x\right) \diagdown
\left( \sigma _{N}\left( x\right) \cap \mathbb{T}\right) }\left\vert
z\right\vert ^{2n}\left\vert z-1\right\vert ^{2}d\mu _{x}\left( z\right) \\
&&+\lim_{n\rightarrow \infty }\int_{\sigma _{N}\left( x\right) \cap \mathbb{T%
}}\left\vert z\right\vert ^{2n}\left\vert z-1\right\vert ^{2}d\mu _{x}\left(
z\right) \\
&=&\int_{\sigma _{N}\left( x\right) \cap \mathbb{T}}\left\vert
z-1\right\vert ^{2}d\mu _{x}\left( z\right) =\int_{\sigma _{N}\left(
x\right) \cap \mathbb{T}\diagdown \left\{ 1\right\} }\left\vert
z-1\right\vert ^{2}d\mu _{x}\left( z\right) .
\end{eqnarray*}%
It follows that $\left\Vert N^{n+1}x-N^{n}x\right\Vert \rightarrow 0$ if and
only if
\begin{equation*}
\mu _{x}\left( \sigma _{N}\left( x\right) \cap \mathbb{T}\diagdown \left\{
1\right\} \right) =0.
\end{equation*}%
By Lemma 4.3 the sequence $\left\{ N^{n}x\right\} $ converges strongly if
and only if
\begin{equation*}
P\left( \sigma _{N}\left( x\right) \cap \mathbb{T}\diagdown \left\{
1\right\} \right) x=0.
\end{equation*}%
By (5.2),
\begin{equation*}
\lim_{n\rightarrow \infty }N^{n}x=\lim_{n\rightarrow \infty }\frac{1}{n}%
\sum_{i=0}^{n-1}N^{i}x=P\left( \left\{ 1\right\} \right) x.
\end{equation*}
\end{proof}

Let $W^{\ast }\left( N\right) $ be the von Neumann algebra generated by $N.$
Recall that $x\in H$ is a \textit{separating vector} for $N$ if the only
operator $A$ in $W^{\ast }\left( N\right) $ such that $Ax=0$ is $A=0$. As is
known \cite[Ch.IX, Section 8.1]{4}, each normal operator has a separating
vector. If $x\in H$ is a separating vector for $N,$ then the spectral
measure of $N$ and the measure $\mu _{x}$ are mutually absolutely continuous
\cite[Ch.IX, Proposition 8.3]{4}, where $\mu _{x}$ is defined by (5.1).

\begin{corollary}
If $x$ is a separating vector for $N,$ then the sequence $\left\{
N^{n}x\right\} $ converges strongly if and only if
\begin{equation}
P\left( \sigma _{N}\left( x\right) \cap \mathbb{T}\diagdown \left\{
1\right\} \right) =0.  \label{5.3}
\end{equation}
\end{corollary}

Now, let $K$ be a compact subset of $\overline{\mathbb{D}}$ such that $1\in
K $ and let $\nu $ be a regular positive Borel measure in $%
%TCIMACRO{\U{2102} }%
%BeginExpansion
\mathbb{C}
%EndExpansion
$ with support $K.$ Define the operator $N$ on $L^{2}\left( K,\nu \right) $
by $Nf=zf.$ Then, $N$ is a normal contraction on $L^{2}\left( K,\nu \right) $
and $\sigma \left( N\right) =K.$ Moreover,
\begin{equation*}
P\left( \Delta \right) f=\chi _{\Delta }f\text{, \ }\forall f\in L^{2}\left(
K,\nu \right) ,
\end{equation*}%
where $\chi _{\Delta }$ is the characteristic function of $\Delta .$ It can
be seen that the identity one function $\mathbf{1}$ on $K$ is a separating
vector for $N$ and $\sigma \left( N\right) =\sigma _{N}\left( \mathbf{1}%
\right) .$ By (5.3), the sequence $\left\{ N^{n}\mathbf{1}\right\} $
converges strongly if and only if $\chi _{\sigma _{N}\left( \mathbf{1}%
\right) \cap \mathbb{T}}=\chi _{\left\{ 1\right\} }$ or $\sigma _{N}\left(
\mathbf{1}\right) \cap \mathbb{T}=\left\{ 1\right\} .$


\begin{thebibliography}{99}
\bibitem{1} J. Barria and P. R. Halmos, Asymptotic Toeplitz operators,
Trans. Amer. Math. Soc. 273(1982), 621-630.

\bibitem{2} B. Beauzamy, Introduction to Operator Theory and Invariant
Subspaces, North Holland, Amsterdam, 1988.

\bibitem{3} A. Brown and P. R. Halmos, Algebraic properties of Toeplitz
operators, J. Reine Angew. Math. 213(1963/1964), 89-102.

\bibitem{4} J. B. Conway, A Course in Functional Analysis, Grad. Texts in
Math. Springer-Verlag, 1985.

\bibitem{5} R. G. Douglas, Banach Algebra Techniques in Operator Theory,
Academic Press, New York, 1972.

\bibitem{6} N. Dunford and J. T. Schwartz, Linear Operators III (Russian),
Mir, Moscow, 1974.

\bibitem{7} T. Eisner,\textit{\ }Stability of Operators and Operator
Semigroups, Oper. Theory Adv. Appl., Birkh\"{a}user, Basel, Vol. 209, 2010.

\bibitem{8} A. Feintuch, On asymptotic Toeplitz and Hankel operators, In the
Gohberg anniversary collection, Oper. Theory Adv. Appl. 41(1989), 241-254.

\bibitem{9} T. Gamelin and J. Garnett, Uniform approximation to bounded
analytic functions, Rev. Un. Math. Argentina, 25(1970), 87-94.

\bibitem{10} I.B. Jung, E. Ko and C. Pearcy, A note on the spectral mapping
theorem, Kyungpook Math. J. 47(2007), 77-79.

\bibitem{11} Y. Katznelson and L. Tzafriri, On power bounded operators, J.
Funct. Anal. 68(1986), 313-328.

\bibitem{12} U. Krengel, Ergodic Theorems, Walter de Gruyter, Berlin, New
York, 1985.

\bibitem{13} K.B. Laursen and M.M. Neumann,\textit{\ }An Introduction to
Local Spectral Theory, Oxford, Clarendon Press, 2000.

\bibitem{14} G. Lumer and M. Rosenblum, Linear operator equations, Proc.
Amer. Math. Soc. 10(1959), 32-41.

\bibitem{15} R. A. Martinez-Avenda\~{n}o, Essentially Hankel operators, J.
London Math. Soc. 66(2002), 741-752.

\bibitem{16} R. A. Martinez-Avenda\~{n}o and P. Rosenthal, An Introduction
to Operators on the Hardy-Hilbert Space, Grad. Texts in Math. 237, Springer,
2007.

\bibitem{17} H. S. Mustafayev, \textit{\ }Asymptotic behavior of
polynomially bounded operators, C. R. Acad. Sci. Paris, Ser. I. 348(2010),
517-520.

\bibitem{18} H. S. Mustafayev, Growth conditions for conjugate orbits of
operators on Banach spaces, J. Oper. Theory, 74(2015), 281-306.

\bibitem{19} B. Sz.-Nagy and C. Foias,\textit{\ }Harmonic Analysis of
Operators on Hilbert Space (Russian), Mir, Moscow, 1970.

\bibitem{20} F. Nazarov and J.H. Shapiro, On the toeplitzness of composition
operators, Complex Var. Elliptic Equ. 52(2007), 193-210.

\bibitem{21} N. K. Nikolski, Treatise on the Shift Operator (Russian),
Nauka, Moscow, 1980.
\end{thebibliography}
\end{document}